\documentclass[oneside,english]{amsart}
\usepackage[T1]{fontenc}
\usepackage[latin9]{inputenc}
\usepackage{mathrsfs}
\usepackage{amstext}
\usepackage{amsthm}
\usepackage{amssymb}
\usepackage[all]{xy}

\makeatletter
\numberwithin{equation}{section}
\numberwithin{figure}{section}
\theoremstyle{plain}
\newtheorem{thm}{\protect\theoremname}
\theoremstyle{remark}
\newtheorem{rem}[thm]{\protect\remarkname}
\theoremstyle{plain}
\newtheorem{prop}[thm]{\protect\propositionname}
\theoremstyle{plain}
\newtheorem{lem}[thm]{\protect\lemmaname}
\theoremstyle{plain}
\newtheorem{cor}[thm]{\protect\corollaryname}
\theoremstyle{definition}
\newtheorem{defn}[thm]{\protect\definitionname}
\theoremstyle{definition}
\newtheorem{example}[thm]{\protect\examplename}

\makeatother

\usepackage{babel}
\providecommand{\corollaryname}{Corollary}
\providecommand{\definitionname}{Definition}
\providecommand{\examplename}{Example}
\providecommand{\lemmaname}{Lemma}
\providecommand{\propositionname}{Proposition}
\providecommand{\remarkname}{Remark}
\providecommand{\theoremname}{Theorem}

\begin{document}
\title{On (pre)-approach spaces within convergence approach spaces}
\author{Frédéric Mynard}
\email{fmynard@njcu.edu}
\address{Department of Mathematics, 2039 JF Kennedy Blvd, Jersey City, NJ 07305}

\maketitle
\global\long\def\G{\mathcal{G}}%
 
\global\long\def\F{\mathcal{F}}%
 
\global\long\def\H{\mathcal{H}}%
\global\long\def\Z{\mathcal{Z}}%
 
\global\long\def\L{\mathcal{L}}%
\global\long\def\U{\mathcal{U}}%
\global\long\def\W{\mathcal{W}}%
 
\global\long\def\E{\mathcal{E}}%
\global\long\def\B{\mathcal{B}}%
 
\global\long\def\A{\mathcal{A}}%
\global\long\def\D{\mathcal{D}}%
\global\long\def\O{\mathcal{O}}%
 
\global\long\def\N{\mathcal{N}}%
 
\global\long\def\X{\mathcal{X}}%
 
\global\long\def\lm{\lim\nolimits}%
 
\global\long\def\then{\Longrightarrow}%
\global\long\def\BaseD{\operatorname{B}_{\mathbb{D}}}%

\global\long\def\V{\mathcal{V}}%
\global\long\def\C{\operatorname{C}}%
\global\long\def\R{\operatorname{R}}%
\global\long\def\adh{\operatorname{adh}}%
\global\long\def\Seq{\operatorname{Seq}}%
\global\long\def\intr{\operatorname{int}}%
\global\long\def\cl{\operatorname{cl}}%
\global\long\def\inh{\operatorname{inh}}%
\global\long\def\id{\operatorname{id}}%
\global\long\def\diam{\operatorname{diam}\ }%
\global\long\def\card{\operatorname{card}}%
\global\long\def\T{\operatorname{T}}%
\global\long\def\S{\operatorname{S}}%
\global\long\def\I{\operatorname{I}}%
\global\long\def\AdhD{\operatorname{A}_{\mathbb{D}}}%
\global\long\def\UD{\operatorname{U}_{\mathbb{D}}}%
\global\long\def\K{\operatorname{K}}%
\global\long\def\LD{\operatorname{L}_{\mathbb{D}}}%
\global\long\def\BRD{\operatorname{B}_{\mathcal{R}(\mathbb{D})}}%
\global\long\def\Dis{\operatorname{Dis}}%
\global\long\def\Cha{\operatorname{Cha}}%
\global\long\def\Reg{\operatorname{Reg}}%

\global\long\def\fix{\operatorname{fix}}%
\global\long\def\Epi{\operatorname{Epi}}%
 
\global\long\def\cont{\mathscr{C}}%
\global\long\def\conv{\mathsf{Conv}}%
\global\long\def\prtop{\mathsf{PrTop}}%
\global\long\def\top{\mathsf{Top}}%
\global\long\def\pstop{\mathsf{PsTop}}%
\global\long\def\Cap{\mathsf{Cap}}%
\global\long\def\prap{\mathsf{PrAp}}%
\global\long\def\ap{\mathsf{Ap}}%
\global\long\def\psap{\mathsf{PsAp}}%
\global\long\def\Cconv{\mathsf{C^{Conv}}}%

\begin{abstract}
The purpose of this note is to illustrate a parallel between (pre)topologies
when seen among convergence spaces and (pre)approach spaces when seen
among convergence approach spaces, that appears to be a more complete
parallel than in the traditional presentation of these approach structures.
This sheds some light on the reflector from the category of convergence
approach spaces to that of approach spaces and even on the structure
of approach spaces as such. This point of view allows for a characterization
of approach spaces among convergence approach spaces represented as
pointfree convergence frames as in \cite{myn.ptfreeAP}.
\end{abstract}

\section{Introduction and preliminaries}

The purpose of this note is to illustrate a parallel between (pre)topologies
when seen among convergence spaces and (pre)approach spaces when seen
among convergence approach spaces, that appears to be a more complete
parallel than in the traditional presentation of these approach structures.
This sheds some light on the reflector from the category of convergence
approach spaces to that of approach spaces and even on the structure
of approach spaces as such. To spell this out, let us walk through
some somewhat lengthy preliminaries (because a very important aspect
is a reframing of the traditional presentation of convergence approach
space, we take some pain in doing this reframing in some details,
even for some known results).

\subsection{Convergence spaces}

If $X$ is a set, we denote its powerset by $\mathbb{P}X$ and its
set of finite subsets by $[X]^{<\infty}$. The set of (set-theoretic)
filters on $X$ is denoted $\mathbb{F}X$ and the set of ultrafilters
on $X$ is denoted $\mathbb{U}X$. If $\A\subset\mathbb{P}X$ is a
family of subsets of $X$, we will consider the following associated
families: its \emph{isotone hull}
\[
\A^{\uparrow}=\{B\subset X:\exists A\in\A,A\subset B\},
\]
its \emph{grill
\[
\A^{\#}=\left\{ B\subset X:\forall A\in\A,B\cap A\neq\emptyset\right\} ,
\]
}its \emph{completion by finite intersections
\[
\A^{\cap}=\{\bigcap\B:\B\in[\A]^{<\infty}\}.
\]
}

When $\A=\{A\}$, we abridge $A^{\uparrow}$ for the principal filter
$\{A\}^{\uparrow}$ of $A$. In particular, if $x\in X$ then $\{x\}^{\uparrow}$
denotes the principal ultrafilter generated by $x$. 

If $\A,\B\subset\mathbb{P}X$, we say that $\A$ and $\B$ \emph{mesh},
in symbols $\A\#\B$, if $\A\subset\B^{\#}$, equivalently, $\B\subset\A^{\#}$.
We say that $\A$ \emph{is finer than $\B$ }or that $\B$ \emph{is
coarser than $\A$, }in symbols\emph{ $\A\geq\B$, }if for every $B\in\B$
there is $A\in\A$ with $A\subset B$, that is, if $\B\subset\A^{\uparrow}$.
For every filter $\F$ on $X$, $\beta(\F)=\{\U\in\mathbb{U}X:\U\geq\F\}$
denotes the set of finer ultrafilters. When $\F=A^{\uparrow}$ is
principal, we abridge to $\beta A=\{\U\in\mathbb{U}X:A\in\U\}$.

A \emph{convergence }$\xi$ on a set $X$ is a relation between points
of $X$ and filters on $X$, denoted $x\in\lm_{\xi}\F$ if $(x,\F)\in\xi$
and interpreted as $x$ is a limit point for $\F$ in $\xi$, satisfying
\begin{equation}
\tag{centered}x\in\lm_{\xi}\{x\}^{\uparrow}\label{eq:centeredconv}
\end{equation}
for every $x\in X$ and 
\begin{equation}
\tag{monotone}\F\leq\G\then\lm_{\xi}\F\subset\lm_{\xi}\G,\label{eq:convmontone}
\end{equation}
for all $\F,\G\in\mathbb{F}X$. The pair $(X,\xi)$ is then called
a \emph{convergence space.}

A map $f:X\to Y$ between two convergence spaces $(X,\xi)$ and $(Y,\tau)$
is \emph{continuous}, in symbols $f\in\cont(\xi,\tau)$, if $f(x)\in\lm_{\tau}f[\F]$
whenever $x\in\lm_{\xi}\F$, where $f[\F]=\{f(F):F\in\F\}^{\uparrow}$
is the image filter. Let $\conv$ denote the category of convergence
spaces and continuous maps. If $\xi,\sigma$ are two convergences
on $X$, we say that $\xi$ is \emph{finer than $\sigma$ }or that
$\sigma$ is \emph{coarser than $\xi$ }if the identity map $\id_{X}\in\cont(\xi,\sigma)$,
that is, if $\lim_{\xi}\F\subset\lim_{\sigma}\F$ for every filter
$\F$ on $X$.

If a convergence additionally satisfies
\begin{equation}
\tag{finite depth}\lm_{\xi}(\F\cap\G)=\lm_{\xi}\F\cap\lm_{\xi}\G\label{eq:fintedepth}
\end{equation}
 for every $\F,\G\in\mathbb{F}X$, we say that $\xi$ \emph{has finite
depth} (\footnote{Many authors include (\ref{eq:fintedepth}) in the axioms of a convergence
space. Here we follow \cite{DM.book}.}). A convergence satisfying the stronger condition that 
\begin{equation}
\tag{pretopology}\lm_{\xi}(\bigcap_{\D\in\mathbb{D}}\D)=\bigcap_{\D\in\mathbb{D}}\lm_{\xi}\D\label{eq:pretopdeep}
\end{equation}
 for every $\mathbb{D}\subset\mathbb{F}X$ is called a \emph{pretopology}
(and $(X,\xi)$ is a \emph{pretopological space}). 

A subset $A$ of a convergence space $(X,\xi)$ is \emph{$\xi$-open
}if 
\[
\lm_{\xi}\F\cap A\neq\emptyset\then A\in\F,
\]
and $\xi$-\emph{closed} if it is closed for limits, that is,
\[
A\in\F\then\lm_{\xi}\F\subset A,
\]
equivalently,
\[
A\in\F^{\#}\then\lm_{\xi}\F\subset A.
\]

Let $\O_{\xi}$ denote the set of open subsets of $(X,\xi)$ and let
$\O_{\xi}(x)=\{U\in\O_{\xi}:x\in U\}$. Similarly, let $\mathcal{C}_{\xi}$
denote the set of closed subsets of $(X,\xi)$. It turns out that
$\O_{\xi}$ is a topology on $X$\emph{. }Moreover,\emph{ }a topology
$\tau$ on a set $X$ determines a convergence $\xi_{\tau}$ on $X$
by
\[
x\in\lm_{\xi_{\tau}}\F\iff\F\geq\N_{\tau}(x),
\]
where $\N_{\tau}(x)$ denotes the neighborhood filter of $x$ for
$\tau$. In turn, $\xi_{\tau}$ completely determines $\tau$ because
$\tau=\O_{\xi_{\tau}}$, so that we do not distinguish between $\tau$
and $\xi_{\tau}$ and identify topologies with special convergences.
Moreover, a convergence $\xi$ on $X$ determines the topology $\O_{\xi}$
on $X$ which turns out to be the finest among the topologies on $X$
that are coarser than $\xi$. We call it the \emph{topological modification
of $\xi$ }and denote it $\T\xi$. Hence the category $\top$ of topological
spaces and continuous maps is a concretely reflective subcategory
of $\conv$. Let $\cl_{\xi}$ denote the closure operator in $\T\xi$.

A convergence $\xi$ determines an \emph{adherence} operator on filters,
namely,
\begin{equation}
\adh_{\xi}\F=\bigcup_{\G\#\F}\lm_{\xi}\G=\bigcup_{\G\geq\F}\lm_{\xi}\G=\bigcup_{\U\in\beta(\F)}\lm_{\xi}\U.\label{eq:adherence}
\end{equation}

A convergence $\xi$ on $X$ is a \emph{pseudotopology }if 
\begin{equation}
\tag{pseudotopology}\lm_{\xi}\F=\bigcap_{\U\in\beta(\F)}\lm_{\xi}\U\label{eq:pseudotop}
\end{equation}
 for every filter $\F\in\mathbb{F}X$. The full subcategory of $\conv$
formed by pseudotopological spaces (and continuous maps), denoted
$\pstop$, is concretely reflective and the reflector $\S$ can be
described on objects by
\begin{equation}
\lm_{\S\xi}\F=\bigcap_{\U\in\beta(\F)}\lm_{\xi}\U=\bigcap_{\G\#\F}\adh_{\xi}\G,\label{eq:Sreflector}
\end{equation}
so that adherences determine the pseudotopological reflection.

When restricted to principal filters, (\ref{eq:adherence}) defines
a \emph{principal adherence operator }$\adh_{\xi}:\mathbb{P}X\to\mathbb{P}X$
given by
\[
\adh_{\xi}A=\adh_{\xi}\{A\}^{\uparrow}=\bigcup_{A\in\G^{\#}}\lm_{\xi}\G=\bigcup_{A\in\G}\lm_{\xi}\G=\bigcup_{\U\in\beta A}\lm_{\xi}\U.
\]

The full subcategory of $\conv$ formed by pretopological spaces (and
continuous maps), denoted $\prtop$, is concretely reflective and
the reflector $\S_{0}$ can be described on objects by
\begin{equation}
\lm_{\S_{0}\xi}\F=\bigcap_{A\in\F^{\#}}\adh_{\xi}A,\label{eq:S0adh}
\end{equation}
so that the principal adherence determines the pretopological reflection.
In general,
\[
\adh_{\xi}A\subset\cl_{\xi}A
\]
but not conversely. In contrast to $\cl_{\xi}$, the principal adherence
is in general non-idempotent because $\adh_{\xi}A$ need not be closed.
A pretopology is a topology if and only if its principal adherence
operator is idempotent, in which case $\adh_{\xi}=\cl_{\xi}$. Moreover,
the reflector $\T$ from $\conv$ onto $\top$ is given by
\begin{equation}
\lm_{\T\xi}\F=\bigcap_{A\in\F^{\#}}\cl_{\xi}A.\label{eq:Tclosure}
\end{equation}

We refer the reader to \cite{DM.book} for a systematic study of convergence
spaces.

\subsection{Convergence approach spaces }

Traditionally, a \emph{convergence approach space }$(X,\lambda)$
is given by a function $\lambda:\mathbb{F}X\to[0,\infty]^{X}$ satisfying
\[
\F\leq\G\then\lambda(\F)\geq\lambda(\G)
\]
for the pointwise order of $[0,\infty]^{X}$ induced by the usual
order on $[0,\infty]$, and 
\[
\lambda(\{x\}^{\uparrow})(x)=0,
\]
for every $x\in X$. The interpretation is that $\lambda(\F)(x)$
measures how well $\F$ converges to $x$, with $\lambda(\F)(x)=0$
indicating full convergence while $\lambda(\F)(x)=\infty$ indicating
no convergence. Additionally, an analog of (\ref{eq:fintedepth})
is usually required, namely that
\begin{equation}
\lambda(\F\cap\G)=\lambda(\F)\vee\lambda(\G),\label{eq:capfinitedepth}
\end{equation}
for every $\F,\G\in\mathbb{F}X$, though we will treat it as an additional
property just like in $\conv$. 

A \emph{contraction} is a function $f:(X,\lambda_{X})\to(Y,\lambda_{Y})$
satisfying 
\[
\lambda_{Y}(f[\F])\circ f\leq\lambda_{X}(\F)
\]
 for the pointwise order induced by that of $[0,\infty]$, for every
$\F\in\mathbb{F}X$. 

In order to more easily draw a complete parallel with $\top$, $\prtop$
and $\conv$, it is more convenient to consider the opposite order
on $[0,\infty]$ so that $0$ is the greatest element and $\infty$
is the smallest element. It is then more intuitive to think of the
values of $\lambda(\F)$ taken in a complete lattice $V$ where $0$
and $1$ denote the least and greatest elements, where $\lambda(\F)(x)=1$
means full convergence and $\lambda(\F)(x)=0$ means no convergence,
and for the purpose of this note, we may restrict ourselves to $V=[0,1]$.
We will also use a multiplicative notation $\otimes$ for the operation
on $[0,1]\sim[0,\infty]^{op}$ instead of the additive traditional
operation in $[0,\infty]$. Indeed, $-\ln:([0,1],\times)\to([0,\infty],+)^{op}$
is a quantale isomorphism. Hence, we prefer to reframe the structure
as a monotone map $\lambda:\mathbb{F}X\to V^{X}$ for the pointwise
order, with 
\begin{equation}
\lambda(\{x\}^{\uparrow})(x)=1\label{eq:CAPcentered}
\end{equation}
for all $x\in X$ and morphisms $f:(X,\lambda_{X})\to(Y,\lambda_{Y})$
satisfying 
\[
\lambda_{Y}(f[\F])\circ f\geq\lambda_{X}(\F)
\]
for every $\F\in\mathbb{F}X$. Let $\cont(X,Y)$ denote the set of
all such morphisms. Let $\Cap$ denote the resulting category with
these morphisms and convergence approach spaces as objects.
\begin{rem}
Though several results here can be easily extended to more general
quantales of values $V$, we will assume in this paper that $V=[0,1]$
or $V=[0,\infty]^{op}$ where the quantale operation $\otimes$ is
$\times$ in the former case and $+$ in the latter. The reader interested
in a full fledged theory of convergence for general quantales $V$
is referred to the extensive work of M.M. Clementino, D. Hofmann,
W. Tholen and their co-authors, e.g. \cite{hofmann2014monoidal,lai2017quantale},
though most of this work allows for a unified handling of categories
with a built-in diagonality, like that of topological spaces, that
of approach spaces, that of closure spaces and more, but rarely ventures
into analogs of convergence spaces. On the other hand, G. J\"ager's
work includes extensive studies of various kinds of quantale-valued
convergence spaces, e.g. \cite{jager2012convergence,jager2018quantale,jager2019quantale,jager2021quantale},
though generally at a higher level of generality involving the convergence
of different types of stratified filters. 
\end{rem}

Note that $\Cap$ is a topological category, where initial structures
are as follows: If $f_{i}:X\to(Y_{i},\lambda_{i})$ then the $\Cap$-initial
structure on $X$ is given by 
\begin{equation}
\lambda_{X}(\F)(x)=\bigwedge_{i\in I}\lambda_{i}(f_{i}[\F])(f(x)).\label{eq:initialCAP}
\end{equation}

Mirroring (\ref{eq:adherence}), if $(X,\lambda)$ is a convergence
approach space and $\F\in\mathbb{F}X$ then the \emph{adherence function
of $\F$ }is the function $\adh_{\lambda}\F:X\to V$ (or $\adh\F$
if $\lambda$ is unambiguous) defined by
\[
\adh_{\lambda}\F(\cdot)=\bigvee_{\G\#\F}\lambda(\G)(\cdot)=\bigvee_{\U\in\beta(\G)}\lambda\U(\cdot),
\]
and in particular, if $A\subset X$,
\[
\adh_{\lambda}A(\cdot):=\adh_{\lambda}\{A\}^{\uparrow}(\cdot)=\bigvee_{A\in\F^{\#}}\lambda\F(\cdot)=\bigvee_{\U\in\beta A}\lambda\U(\cdot).
\]

\begin{rem}
Note that in an approach space, $\adh A(\cdot)$ has traditionally
been denoted $\delta(\cdot,A)$ and thought of as the distance from
$A$ (valued in $[0,\infty]$ with $0$ being our $1$). We prefer
to think of it as a measure of adherence, and as measure of closure
when we have an analog of idempotency, which is an essential aspect
of this paper to be clarified in Section \ref{sec:Closure}.
\end{rem}

Analogously to (\ref{eq:pretopdeep}), a convergence approach space
$(X,\lambda)$ is a called a \emph{pre-approach space }(e.g., \cite{Lowen97})
if for any $\mathbb{D}\subset\mathbb{F}X$,
\begin{equation}
\lambda(\bigcap_{\D\in\mathbb{D}}\D)(\cdot)=\bigwedge_{\D\in\mathbb{D}}\lambda\D(\cdot).\tag{PrAp}\label{eq:preAP}
\end{equation}

If (\ref{eq:preAP}) is only true for subsets $\mathbb{D}$ of $\mathbb{U}X$
of the form $\beta(\F)$, that is, if 
\begin{equation}
\tag{PsAp}\lambda(\F)(\cdot)=\bigwedge_{\U\in\beta(\F)}\lambda\U(\cdot)\label{eq:psap}
\end{equation}
for every $\F\in\mathbb{F}X$, then $(X,\lambda)$ is a called a \emph{pseudo-approach
space }(e.g.,\cite{lowe88}), generalizing pseudotopologies (compare
(\ref{eq:pseudotop})).

Note that:
\begin{prop}
\label{prop:preimageadh} If $(X,\lambda_{X})$ and $(Y,\lambda_{Y})$
are two convergence approach spaces and $f\in\cont(X,Y)$ and $\G\in\mathbb{F}Y$
then 
\begin{equation}
\adh_{\lambda_{X}}(f^{-}[\G])\leq\adh_{\lambda_{Y}}\G\circ f.\label{eq:adhcontraction}
\end{equation}

In particular, if $A\subset Y$, 
\begin{equation}
\adh_{\lambda_{X}}(f^{-}A)\leq\adh_{\lambda_{Y}}A\circ f.\label{eq:princadhcontraction}
\end{equation}
\end{prop}

\begin{proof}
We have
\[
\adh_{\lambda_{X}}(f^{-}[\G])(x)=\bigvee_{\U\in\beta(f^{-}[\G])}\lambda_{X}\U(x)
\]
 and $\lambda_{X}(\U)(x)\leq\lambda_{Y}(f[\U])(f(x))$ because $f\in\cont(X,Y)$,
with 
\[
f[\U]\in\beta(f(f^{-}[\G]))\subset\beta(\G),
\]
 so that 
\[
\adh_{\lambda_{X}}(f^{-}[\G])(x)\leq\bigvee_{\U\in\beta(f^{-}[\G])}\lambda_{Y}(f[\U])(f(x))\leq\bigvee_{\W\in\beta(\G)}\lambda_{Y}\W(f(x))=\adh_{\lambda_{Y}}\G(f(x)).
\]
\end{proof}
\begin{thm}
\label{thm:praprefl}\cite[Theorem 1 and Proposition 2]{mynardmeasureCAP}
\emph{(}\footnote{Recall that we assume $V=([0,1],\times)$ or equivalently $([0,\infty]^{op},+)$.}\emph{)}
A convergence approach space $(X,\lambda)$ is a pre-approach space
if and only if 
\begin{equation}
\lambda(\F)(\cdot)=\bigwedge_{A\in\F^{\#}}\adh_{\lambda}A(\cdot)\label{eq:preAPadh-1}
\end{equation}
for every $\F\in\mathbb{F}X$. More generally, given a convergence
approach space $(X,\lambda)$, the map
\[
\lambda_{\S_{0}}(\F)=\bigwedge_{A\in\F^{\#}}\adh_{\lambda}A
\]
defines the pre-approach reflection of $(X,\lambda)$ and 
\begin{equation}
\adh_{\lambda}A=\adh_{\lambda_{\S_{0}}}A\label{eq:sameadhS0}
\end{equation}
for every $A\subset X$.
\end{thm}

Theorem \ref{thm:praprefl} allows for a converse of sort (which I
believe to be folklore, but that I prove as I do not have a reference
for it) to Proposition \ref{prop:preimageadh}:
\begin{prop}
Let $(Y,\lambda_{Y})$ be a convergence approach space and $f:X\to Y$.
If $(Y,\lambda_{Y})$ is a pseudo-approach space and \emph{(\ref{eq:adhcontraction})}
for every $\G\in\mathbb{F}Y$ then $f\in\cont(X,Y).$ If $(Y,\lambda_{Y})$
is a pre-approach space and \emph{(\ref{eq:princadhcontraction})}
for every $A\subset Y$, then $f\in\cont(X,Y)$.
\end{prop}

\begin{proof}
If (\ref{eq:adhcontraction}) for every $\G\in\mathbb{F}Y$, then
in particular, for every ultrafilter $\U$ on $X$, $\G=f[\U]\in\mathbb{U}Y$
so that 
\[
\adh_{\lambda_{X}}(f^{-}[\G])=\adh_{\lambda_{X}}(f^{-}[f[\U]])\leq\adh_{\lambda_{Y}}f[\U]\circ f
\]
and $\adh_{\lambda_{Y}}f[\U]=\lambda_{Y}(f[\U])$ because $f[\U]\in\mathbb{U}Y$.
Moreover, $f^{-}[f[\U]]\leq\U$ so that 
\[
\lambda_{X}\U\leq\adh_{\lambda_{X}}(f^{-}[f[\U]])\leq\lambda_{Y}(f[\U])\circ f
\]
for every ultrafilter $\U$ on $X$. If $\F\in\mathbb{F}X$, 
\[
\lambda_{X}(\F)\leq\bigwedge_{\U\in\beta(\F)}\lambda_{X}(\U)\leq\bigwedge_{\U\in\beta(\F)}\lambda_{Y}(f[\U])\circ f=\lambda_{Y}\Big(\bigwedge_{\U\in\beta(\F)}f[\U]\Big)\circ f=\lambda_{Y}(f[\F])\circ f,
\]
where the first equality follows from (\ref{eq:psap}) in $Y$.

Now, assume $(Y,\lambda_{Y})$ is a pre-approach space and (\ref{eq:princadhcontraction})
for every $A\subset Y$, and let $\F\in\mathbb{F}X$. Since $f^{-}(A)\#\F$
for every $A\#f[\F]$, 
\[
\lambda_{X}(\F)\leq\adh_{\lambda_{X}}(f^{-}A)\leq\adh_{\lambda_{Y}}A\circ f.
\]
Hence, 
\[
\lambda_{X}(\F)\leq\bigwedge_{A\#f[\F]}\adh_{\lambda_{Y}}A\circ f=\lambda_{Y}(f[\F])\circ f
\]
because of (\ref{eq:preAPadh-1}) in $Y$.
\end{proof}
\begin{rem}
In $(V,\otimes)$, for each $v\in V$, the map $v\otimes-:V\to V$
preserves all suprema, hence has a right adjoint $-\oslash v:V\to V$
defined by 
\begin{equation}
v\otimes x\leq y\iff x\leq y\oslash v,\label{eq:adjointtensor-1}
\end{equation}
which preserves all infima. In $V=([0,1],\times)$, 
\begin{equation}
t\oslash x=\begin{cases}
\frac{t}{x}\wedge1 & x\neq0\\
1 & x=0
\end{cases},\label{eq:divin01-1}
\end{equation}
while in $V=([0,\infty]^{op},+)$, $t\oslash x=(t-x)\wedge0$ (note
that here $\wedge$ is the maximum for the usual order of $[0,\infty]$).

Note that (\ref{eq:adjointtensor-1}) yields
\begin{equation}
v\oslash\bigvee_{a\in A}a=\bigwedge_{a\in A}(v\oslash a),\label{eq:oslahofsup}
\end{equation}
for every $v\in V$ and $A\subset V$ and 
\begin{equation}
(v\oslash x)\oslash(v\oslash t)\geq t\oslash x,\label{eq:doublediv}
\end{equation}
for every $t,v,x\in V$.
\end{rem}

Recall that as $V$ is completely distributive, if $\F\in\mathbb{F}X$
and $\phi:X\to V$ then
\begin{equation}
\bigvee_{A\in\F^{\#}}\bigwedge_{x\in A}\phi(x)=\bigwedge_{F\in\F}\bigvee_{x\in F}\phi(x),\label{eq:CD}
\end{equation}
so that in particular, for an ultrafilter $\U\in\mathbb{U}X$, we
have 
\begin{equation}
\bigvee_{U\in\U}\bigwedge_{x\in U}\phi(x)=\bigwedge_{U\in\U}\bigvee_{x\in U}\phi(x).\label{eq:CDultra}
\end{equation}

We will also use the notations
\[
\liminf_{\F}\phi=\bigvee_{F\in\F}\bigwedge_{x\in F}\phi(x)
\]
and 
\[
\limsup_{\F}\phi=\bigwedge_{F\in\F}\bigvee_{x\in F}\phi(x),
\]
where $\F$ is a family of subsets of $X$ and $\phi:X\to V$.

Note that 
\begin{equation}
\F\subset\G\then\liminf_{\F}\phi\leq\liminf_{\G}\phi\label{eq:liminfmonotone}
\end{equation}
and 
\[
\F\subset\G\then\limsup_{\G}\phi\leq\limsup_{\F}\phi.
\]

If $\star$ is a binary operation on $V$ and $f,g\in V^{X}$ then
$f\star g$ is defined pointwise by 
\[
(f\star g)(x)=f(x)\star g(x).
\]
If $g(x)=v$ for all $x$, we write $f\star v$ for $f\star g$. 

For our $V$ the operation $\otimes$ preserves both infima and suprema
and $\oslash$ does as well in its first variable, so that
\begin{equation}
(\limsup_{\F}\phi)\star v=\limsup_{\F}(\phi\star v)\label{eq:limsupstarv}
\end{equation}
and 
\begin{equation}
(\liminf_{\F}\phi)\star v=\liminf_{\F}(\phi\star v)\label{eq:liminfstarv}
\end{equation}
 for these operations. 

If $A\subset X$, where $(X,\lambda)$ is a convergence approach space
and $\epsilon\in V$, let 
\[
A^{(\epsilon)}:=\left\{ x\in X:\adh A(x)\geq\epsilon\right\} =(\adh A)^{-1}[\epsilon,1].
\]
 A pre-approach space $(X,\lambda)$ is an \emph{approach space }(e.g,
\cite{Lowen89,AP.book,indextheory}) if 
\begin{equation}
\adh A(x)\geq\adh A^{(\epsilon)}(x)\otimes\epsilon\label{eq:diagonaladh}
\end{equation}
for every $A\subset X$, every $x\in X$ and every $\epsilon\in V$
(\footnote{In its traditional form valued in $([0,\infty],+)$, 
\[
\delta(A,x)\leq\delta(A^{(\epsilon)},x)+\epsilon.
\]
}).

Let $\psap$, $\prap$ and $\ap$ denote the full subcategories of
$\Cap$ of pseudo-approach, pre-approach and approach spaces respectively.

Recall as well (e.g., \cite{Lowen88}) that $c:\Cap\rightarrow\conv$
and $r:\Cap\rightarrow\conv$ defined on objects by 
\begin{eqnarray*}
x & \in & \lim\nolimits_{c(\lambda)}\mathcal{F\Longleftrightarrow\lambda F}(x)=1\\
x & \in & \lim\nolimits_{r(\lambda)}\mathcal{F\Longleftrightarrow\lambda F}(x)>0
\end{eqnarray*}
extend to a concrete coreflector and a concrete reflector respectively.
The functors $c$ and $r$ restrict to coreflectors and reflectors
from $\psap$ to $\pstop$, $\mathbf{\prap}$ to $\prtop$ and $\ap$
to $\top$.

\subsection{Main results}

There is a well known approach structure on $V$, which, though classical,
we revisit in some details in Section \ref{subsec:The-approach-structure}
below. A key observation (Theorem \ref{prop:adhcont}) is that for
this structure (\ref{eq:diagonaladh}) is equivalent to $\adh A\in\cont(X,V)$.
A second observation (Section \ref{sec:Closure}) is that the $\Cap$-analog
of the closed subsets of a convergence space are elements of $\cont(X,V)$
and that $\cont:V^{X}\to\cont(X,V)$ where $\cont(f)$ is the smallest
element of $\cont(X,V)$ larger than $f$ is the lower-hull operator
associated with $(X,\lambda)$ and determines the approach reflection
of $(X,\lambda)$. In fact, special elements suffice: it turns out
that if $\theta_{A}:X\to V$ defined by $\theta_{A}(x)=1$ if $x\in A$
and $\theta_{A}(x)=0$ otherwise denotes the \emph{indicator function
of $A$}, then the functions of the form $\cont(\theta_{A})=\cont(\adh_{A})$
play the role of closure. With the notation $\cl A=\cont(\theta_{A})$,
we obtain a $\Cap$-analog (Theorem \ref{thm:approachreflection})
of (\ref{eq:Tclosure}).

As $\cont(X,V)$ is the natural $\Cap$-analog of $\mathcal{C}_{\xi}$,
$\Cap$ and $\ap$ generalizations of hyperspace theory should be
carried out on this set, rather than on the set $CL(X)$ of closed
subsets of the $\top$-reflection or $\top$-coreflection of an approach
space $(X,\lambda)$, as it has traditionally been done. As $\theta_{A}\in\cont(X,V)$
if and only $A$ is $r(\lambda)$-closed (Theorem \ref{thm:adhcont}),
$CL(X)$ should be chosen to be the set of closed subsets in the $\top$-reflection
of $(X,\lambda)$ so that $CL(X)$ can naturally be identified with
a subset of $\cont(X,V)$. Hence, $\Cap$ and $\ap$ structures on
$\cont(X,V)$--the natural hyperstructures in $\Cap$--also induce
$\Cap$ structures on $CL(X)$ and traditional results using approach
structures in the study of hyperspaces can be recovered. This is the
subject of forthcoming research.

On the other hand, this point of view on approach spaces provides
a natural way to characterize those convergence frames $(V^{X},\lim)$
(in the sense of \cite{FredetJean}) that are images of an approach
space under the functor $V^{(-)}:\Cap\to(\Cconv)^{op}$ introduced
in \cite{myn.ptfreeAP}, thus answering (via Theorem \ref{thm:VAp})
a question left in \cite{myn.ptfreeAP}.

\subsection{The approach structure on $V$\protect\label{subsec:The-approach-structure}}

$V$ can be turned into a convergence approach space (in fact, an
approach space) in a standard way (See e.g., \cite[Example 1.8.33]{AP.book},
\cite[p. 64]{indextheory}) by defining
\begin{eqnarray}
\lambda_{V}(\F)(v) & = & v\oslash(\bigvee_{A\in\F^{\#}}\bigwedge A)=v\oslash\liminf_{\F^{\#}}\id_{V}\label{eq:lambdaV1}
\end{eqnarray}
for every $\F\in\mathbb{F}V$ and $v\in V$. Note that in particular,
$\lambda_{V}(\F)(v)=1$ whenever $v\geq\bigvee_{A\in\F^{\#}}\bigwedge A$
because $v\oslash t=1$ if and only if $v\geq t$. Additionally
\begin{equation}
\lambda_{V}(\F)(v)=\bigwedge_{A\in\F^{\#}}v\oslash\bigwedge A,\label{eq:lambdaV}
\end{equation}
by (\ref{eq:oslahofsup}). Since $(\bigcap_{\D\in\mathbb{D}}\D)^{\#}=\bigcup_{\D\in\mathbb{D}}\D^{\#}$,
we conclude that 
\[
\lambda_{V}(\bigcap_{\D\in\mathbb{D}}\D)(v)=\bigwedge_{\D\in\mathbb{D}}\bigwedge_{A\in\D^{\#}}v\oslash\bigwedge A=\bigwedge_{\D\in\mathbb{D}}\lambda_{V}\D(v),
\]
so that $\lambda_{V}$ is a pre-approach limit. This is actually an
approach limit, as is well known. We will nevertheless provide proofs
in order to build familiarity with $(V,\lambda_{V})$. 

For $V$ completely distributive (such as $V=[0,1]$),
\begin{eqnarray*}
\lambda_{V}(\F)(v) & = & v\oslash(\bigvee_{A\in\F^{\#}}\bigwedge A)=\bigwedge_{A\in\F^{\#}}v\oslash\bigwedge A\\
 & = & v\oslash(\bigwedge_{F\in\F}\bigvee F)=v\oslash\limsup_{\F}\id_{V}.
\end{eqnarray*}

\begin{lem}
\label{lem:adhinV} Let $v\in V$ and $A$ be a non-empty subset of
$V$. Then 
\[
\adh_{V}A(v)=v\oslash\bigwedge A.
\]
\end{lem}

Recall that here $V=([0,1],\times)$, equivalently, $V=([0,\infty]^{op},+)$,
in which case this is classical. See, e.g., \cite[p.63]{indextheory},
where the structure is given in terms of $\delta_{\mathbb{P}}(x,A)$
which is our $\adh_{V}A(x)$. The corresponding $\lambda_{V}$ ($\lambda_{\mathbb{P}}$
in the notations of \cite{indextheory}) is then derived from general
formulas. Here we started from $\lambda_{V}$ and include a proof
of the formula for $\adh_{V}A$ and of Corollary \ref{cor:Vap} for
the sake of developing familiarity with $(V,\lambda_{V})$.
\begin{proof}
In view of (\ref{eq:lambdaV}), for every $\F$ with $A\in\F^{\#}$,
$\lambda_{V}(\F)(v)\leq v\oslash\bigwedge A$ so that 
\[
\adh_{V}A(v)=\bigvee_{\F\#A}\lambda_{V}\F(v)\leq v\oslash\bigwedge A.
\]

Moreover, if 
\begin{eqnarray*}
\adh_{V}A(v) & = & \bigvee_{\U\in\beta A}\lambda_{V}(\U)(v)<v\oslash\bigwedge A
\end{eqnarray*}
 then 
\[
\lambda_{V}(\U)(v)=v\oslash(\bigvee_{U\in\U}\bigwedge U)=\bigwedge_{U\in\U}v\oslash\bigwedge U<v\oslash\bigwedge A
\]
 for every $\U\in\beta A$. Hence for every $\U\in\beta A$, there
is $U_{\U}\in\U$ with $(v\oslash\bigwedge U_{\U})<v\oslash\bigwedge A$.
Moreover, there is a finite subset $\mathbb{D}$ of $\beta A$ with
$A\subset\bigcup_{\U\in\mathbb{D}}U_{\U}$, hence $\bigwedge A\geq\bigwedge(\bigcup_{\U\in\mathbb{D}}U_{\U})=\bigwedge_{\U\in\mathbb{D}}\bigwedge U_{\U}$.
Moreover, in $[0,1]$ infima of finite sets are minima, so that there
is $\U_{0}\in\mathbb{D}$ with $\bigwedge_{\U\in\mathbb{D}}\bigwedge U_{\U}=\bigwedge U_{\U_{0}}$,
so that by (\ref{eq:doublediv}), 
\[
v\oslash\bigwedge A\leq v\oslash\bigwedge U_{\U_{0}}
\]
 in contradiction to $(v\oslash\bigwedge U_{\U_{0}})<v\oslash\bigwedge A$.
\end{proof}
\begin{cor}
\label{cor:Vap}\cite[Example 1.8.33]{AP.book}, \cite[p. 64]{indextheory}
The convergence approach space $(V,\lambda_{V})$ is an approach space.
\end{cor}

\begin{proof}
Noting that, given (\ref{eq:adjointtensor-1}) and Lemma \ref{lem:adhinV},
if $A\subset V$ then 
\[
A^{(\epsilon)}=\{v\in V:v\oslash\bigwedge A\geq\epsilon\}=\{v\in V:v\geq\bigwedge A\otimes\epsilon\},
\]
we show (\ref{eq:diagonaladh}) in this case, that is,
\begin{equation}
v\oslash\bigwedge A\geq(v\oslash\bigwedge A^{(\epsilon)})\otimes\epsilon.\label{eq:APV}
\end{equation}
Indeed, $\bigwedge A^{(\epsilon)}\geq(\bigwedge A)\otimes\epsilon$
and thus $\bigwedge A^{(\epsilon)}\in A^{(\epsilon)}$ so that $\bigwedge A^{(\epsilon)}\oslash\bigwedge A\geq\epsilon$.
In view of (\ref{eq:doublediv}), with $x=\bigwedge A$ and $t=\bigwedge A^{(\epsilon)}$,
we have 
\[
\left(v\oslash\bigwedge A\right)\oslash(v\oslash\bigwedge A^{(\epsilon)})\geq\bigwedge A^{(\epsilon)}\oslash\bigwedge A\geq\epsilon,
\]
equivalently (\ref{eq:APV}), using (\ref{eq:adjointtensor-1}).
\end{proof}

\section{\protect\label{sec:Closure}Closure}

Up to this point, we just revisited classical facts about the structure
of $V$ from the viewpoint of convergence. Here is now the first key
new observation.
\begin{thm}
\label{prop:adhcont} Let $(X,\lambda)$ be a convergence approach
space and let $A\subset X$. Then $\adh A\in\cont(X,V)$ if and only
if \emph{(\ref{eq:diagonaladh})} for every $\epsilon\in V$.
\end{thm}

\begin{proof}
Assume (\ref{eq:diagonaladh}). Let $\F\in\mathbb{F}X$, $A\subset X$
and $x\in X$. Then 
\begin{eqnarray*}
\lambda_{V}(\adh A[\F])(\adh A(x)) & = & \bigwedge_{B\#\adh A[\F]}\adh A(x)\oslash\bigwedge B\\
 & = & \bigwedge_{(\adh A)^{-1}(B)\#\F}\adh A(x)\oslash\bigwedge B\\
 & \geq & \bigwedge_{(\adh A)^{-1}(B)\#\F}\adh A^{(\wedge B)}(x)
\end{eqnarray*}
by (\ref{eq:diagonaladh}). $B\subset[\bigwedge B,1]$ so that $(\adh A)^{-1}(B)\subset A^{(\wedge B)}$
and thus 

\[
\bigwedge_{(\adh A)^{-1}(B)\#\F}\adh A^{(\wedge B)}(x)\geq\bigwedge_{C\#\F}\adh C(x)\geq\lambda_{X}(\F)(x).
\]
 Thus $\adh A\in\cont(X,V)$.

Conversely, suppose $\adh A\in\cont(X,V)$, so that
\[
\lambda_{V}(\adh A[\F])(\adh A(x))=\bigwedge_{(\adh A)^{-1}(B)\#\F}\adh A(x)\oslash\bigwedge B\geq\lambda(\F)(x),
\]
for every $\F\in\mathbb{F}X$ and $x\in X$, so that in particular,
for $B=[\epsilon,1]$, we obtain that 
\[
\adh A(x)\oslash\epsilon\geq\lambda(\F)(x)
\]
whenever $(\adh A)^{-1}(B)=A^{(\epsilon)}\in\F^{\#}$. Since 
\[
\adh A^{(\epsilon)}(x)=\bigvee_{A^{(\epsilon})\in\F^{\#}}\lambda(\F)(x),
\]
we conclude that 
\[
\adh A(x)\oslash\epsilon\geq\adh A^{(\epsilon)}(x),
\]
equivalently, (\ref{eq:diagonaladh}).
\end{proof}
\begin{cor}
\label{cor:approachadh} A convergence approach space $(X,\lambda)$
is an approach space if and only if it is a pre-approach space and
$\adh A\in\cont(X,V)$ for every $A\subset X$.
\end{cor}

For a general convergence approach space $(X,\lambda)$ and $A\subset X$,
we have seen that $\adh A:X\to V$ may or may not be a morphism. A
function $f\in V^{X}$ belongs to $\cont(X,V)$ if 
\[
\lambda_{V}(f[\F])(f(x))=f(x)\oslash(\bigwedge_{F\in\F}\bigvee f(F))\geq\lambda_{X}(\F)(x),
\]
that is, if
\begin{equation}
f(x)\oslash\limsup_{\F}f\geq\lambda_{X}(\F)(x)\label{eq:contXtoV}
\end{equation}
for every $x\in X$ and $\F\in\mathbb{F}X$. 

Note that in view of (\ref{eq:limsupstarv}),
\begin{prop}
\label{prop:constantoperationpointwise} If $f\in\cont(X,V)$, and
$v\in V$ then $f\otimes v\in\cont(X,V)$.
\end{prop}

\begin{proof}
Since $f\in\cont(X,V)$, 
\[
\lambda_{V}(f[\F])(f(x))=f(x)\oslash\limsup_{\F}f\geq\lambda_{X}(\F)(x),
\]
equivalently, $f(x)\geq$$\lambda_{X}(\F)(x)\otimes\limsup_{\F}f$
and thus 
\[
f(x)\otimes v\geq\lambda_{X}(\F)(x)\otimes(\limsup_{\F}f)\otimes v=\lambda_{X}(\F)(x)\otimes\limsup_{\F}(f\otimes v)
\]
by (\ref{eq:limsupstarv}), so that 
\[
(f(x)\otimes v)\oslash\limsup_{\F}(f\otimes v)\geq\lambda_{X}(\F)(x)
\]
and thus $f\otimes v\in\cont(X,V)$.
\end{proof}
In view of (\ref{eq:contXtoV}), if $S\subset\cont(X,V)$ then the
pointwise infimum (in $V^{X}$) satisfies
\begin{eqnarray*}
\lambda_{V}((\bigwedge_{f\in S}f)[\F])(\bigwedge_{f\in S}f(x)) & = & \left(\bigwedge_{f\in S}f(x)\right)\oslash(\bigwedge_{F\in\F}\bigvee_{t\in F}\bigwedge_{f\in S}f(t))\\
 & \geq & \left(\bigwedge_{f\in S}f(x)\right)\oslash(\bigwedge_{f\in S}\bigwedge_{F\in\F}\bigvee_{t\in F}f(t))
\end{eqnarray*}
because $\bigwedge_{F\in\F}\bigvee_{t\in F}\bigwedge_{f\in S}f(t)\leq\bigwedge_{f\in S}\bigwedge_{F\in\F}\bigvee_{t\in F}f(t)$
and (\ref{eq:oslahofsup}). Hence,
\[
\lambda_{V}((\bigwedge_{f\in S}f)[\F])(\bigwedge_{f\in S}f(x))\geq\bigwedge_{f\in S}\left(f(x)\oslash\limsup_{\F}f\right)\geq\lambda_{X}(\F)(x).
\]
In other words, $\bigwedge_{f\in S}f\in\cont(X,V)$ and thus $\cont(X,V)$
is a complete lattice. The supremum in $\cont(X,V)$ does not coincide
with that in $V^{X}$ in general: it is the pointwise infimum of the
upper bounds in $\cont(X,V)$. Note however, that if $f,g\in\cont(X,V)$
then the pointwise supremum $f\vee g$ in $V^{X}$ belongs to $\cont(X,V)$
(\footnote{\label{fn:finitesup}Indeed, 
\[
\lambda_{V}(f\vee g)[\F](f(x)\vee g(x))=(f(x)\vee g(x))\oslash(\bigwedge_{F\in\F}\bigvee_{t\in F}f(t)\vee g(t))
\]
and $f(x)\vee g(x)=f(x)$ or $f(x)\vee g(x)=g(x)$. In the first case,
\[
\lambda_{V}(f\vee g)[\F](f(x)\vee g(x))=f(x)\oslash(\bigwedge_{F\in\F}\bigvee_{t\in F}f(t)\vee g(t))\geq f(x)\oslash(\bigwedge_{F\in\F}\bigvee_{t\in F}f(t))\geq\lambda(\F)(x)
\]
because $f\in\cont(X,V)$ and similarly in the second case.}).

Moreover, given $f\in V^{X}$, 
\[
\uparrow f\cap\cont(X,V)=\{g\in\cont(X,V):f\leq g\}\neq\emptyset,
\]
for the constant function $1$ belongs to this set. Hence, there exists
the smallest morphism 
\begin{equation}
\cont(f)=\bigwedge\{g\in\cont(X,V):f\leq g\}\label{eq:C(f)}
\end{equation}
that is larger (for the pointwise order) than $f$.

Note that 
\begin{equation}
f\leq g\then\cont(f)\leq\cont(g)\label{eq:Cmonotone}
\end{equation}
for every $f,g\in V^{X}$ because $\cont(g)$ is a morphism greater
than $f$. Moreover,
\begin{prop}
\label{prop:Cregfunctionframe} Let $(X,\lambda)$ be a convergence
approach space and let $\cont:V^{X}\to\cont(X,V)\subset V^{X}$ be
given by \emph{(\ref{eq:C(f)})}. Then $\cont$ is a lower-hull operator
(in the sense of \cite[Def. 1.1.26]{indextheory} \footnote{where the case $V=[0,\infty]^{op}$ with $1=0$ and $0=\infty$ and
$\otimes=+$ is the one treated. Of course, as we are turning the
order of \cite{indextheory} around, it could make sense to call this
\emph{upper-hull operator }instead.}), that is, 
\begin{enumerate}
\item $f\leq\cont(f)$ for all $f\in V^{X}$;
\item $\cont(f\vee g)=\cont(f)\vee\cont(g)$ for all $f,g\in V^{X}$;
\item $\cont(\cont(f))=\cont(f)$ for all $f\in V^{X}$;
\item $\cont(f\otimes v)=\cont(f)\otimes v$ for every $f\in V^{X}$ and
$v\in V$.
\end{enumerate}
\end{prop}

\begin{proof}
(1) and (3) are clearly true. To see (2), note that $f\vee g\leq\cont(f)\vee\cont(g)$
and a finite pointwise supremum of morphism is a morphism, as we have
seen (see Footnote \ref{fn:finitesup}). Hence $\cont(f\vee g)\leq\cont(f)\vee\cont(g)$.
In view of (\ref{eq:Cmonotone}), $\cont(f)\leq\cont(f\vee g)$ and
$\cont(g)\leq\cont(f\vee g)$ so that $\cont(f\vee g)=\cont(f)\vee\cont(g)$. 

To see (4), note that $f(x)\otimes v\leq\cont(f)(x)\otimes v$ and
$\cont(f)(x)\otimes v\in\cont(X,V)$ by Proposition \ref{prop:constantoperationpointwise},
so that $\cont(f\otimes v)\leq\cont(f)\otimes v$. Moreover, if $g\in\cont(X,V)$
and $f\otimes v\leq g$ then $f\leq g\oslash v$ and $g\oslash v\in\cont(X,V)$
by Proposition \ref{prop:constantoperationpointwise}, so that $\cont(f)\leq g\oslash v$,
equivalently, $\cont(f)\otimes v\leq g$. Hence $\cont(f)\otimes v=\cont(f\otimes v)$.
\end{proof}
Note that $\theta_{A}:X\to V$ defined by $\theta_{A}(x)=\begin{cases}
1 & x\in A\\
0 & x\notin A
\end{cases}$ satisfies 
\begin{equation}
A\in\F^{\#}\iff\limsup_{\F}\theta_{A}=1\iff\limsup_{\F}\theta_{A}>0,\label{eq:grilltheta}
\end{equation}
for every $\F\in\mathbb{F}X$ and $A\subset X$. Similarly,
\begin{equation}
A\in\F\iff\liminf_{\F}\theta_{A}=1\iff\liminf_{\F}\theta_{A}>0.\label{eq:belongthetaA}
\end{equation}
More generally, for every $\epsilon\in V\setminus\{1\}$, the map
$\theta_{A}^{\epsilon}:X\to V$ defined by $\theta_{A}^{\epsilon}(x)=\begin{cases}
1 & x\in A\\
\epsilon & x\notin A
\end{cases}$ satisfies
\[
A\in\F^{\#}\iff\limsup_{\F}\theta_{A}^{\epsilon}=1\iff\limsup_{\F}\theta_{A}^{\epsilon}>\epsilon,
\]
and 
\[
A\in\F\iff\liminf_{\F}\theta_{A}^{\epsilon}=1\iff\liminf_{\F}\theta_{A}^{\epsilon}>\epsilon.
\]

\begin{lem}
\label{lem:tauepsilonclosed} If $(X,\lambda)$ is a convergence space,
$A\subset X$ and $\epsilon\in V$ then $\theta_{A}^{\epsilon}\in\cont(X,V)$
if and only if $A$ is $\tau_{\epsilon}$-closed where $x\in\lm_{\tau_{\epsilon}}\F$
if and only if $\lambda(\F)(x)>\epsilon$.
\end{lem}

\begin{proof}
The condition 
\[
\lambda_{V}(\theta_{A}^{\epsilon}[\F])(\theta_{A}^{\epsilon}(x))=\theta_{A}^{\epsilon}(x)\oslash\limsup_{\F}\theta_{A}^{\epsilon}\geq\lambda\F(x)
\]
 is empty if $x\in A$ because $1\oslash v=1$. Suppose $x\notin A$.
If $A\notin\F^{\#}$ then $\limsup_{\F}\theta_{A}^{\epsilon}=\epsilon$
and $\theta_{A}^{\epsilon}(x)\oslash\limsup_{\F}\theta_{A}^{\epsilon}=\epsilon\oslash\epsilon=1\geq\lambda\F(x)$
is satisfied. If $A\in\F^{\#}$ then $\limsup_{\F}\theta_{A}^{\epsilon}=1$
and $\theta_{A}^{\epsilon}(x)\oslash\limsup_{\F}\theta_{A}^{\epsilon}=\epsilon\oslash1=\epsilon$,
so that $\theta_{A}^{\epsilon}\in\cont(X,V)$ if and only if for every
$x\notin A$ and every $\F$ with $A\in\F^{\#}$, $\lambda\F(x)\leq\epsilon$,
that is, $A$ is $\tau_{\epsilon}$-closed.
\end{proof}
\begin{thm}
\label{thm:adhcont} If $(X,\lambda)$ is a convergence space and
$A\subset X$ then
\[
\cont(\adh A)=\cont(\theta_{A}).
\]
Moreover, the following are equivalent:
\begin{enumerate}
\item $\theta_{A}\in\cont(X,V);$ 
\item $\theta_{A}^{\epsilon}\in\cont(X,V)$ for every $\epsilon\in V\setminus\{0\}$;
\item $A$ is $r(\lambda)$-closed. 
\end{enumerate}
\end{thm}

\begin{proof}
As $\theta_{A}\leq\adh A$, $\cont(\theta_{A})\leq\cont(\adh A)$.
Moreover, $adhA\leq\cont(\theta_{A})$ and the equality follows. Indeed,
\[
\adh A(x)=\bigvee_{\U\in\beta A}\lambda\U(x)
\]
and $\theta_{A}(\U)=\{1\}$ so that $\cont(\theta_{A})[\U]=\{1\}^{\uparrow}$
and 
\[
\lambda_{V}(\cont(\theta_{A}))[\U](\cont(\theta_{A})(x))=\lambda_{V}(\{1\}^{\uparrow})(\cont(\theta_{A})(x))=\cont(\theta_{A})(x)\oslash1=\cont(\theta_{A})(x),
\]
so that $\cont(\theta_{A})(x)\geq\lambda\U(x)$ because $\cont(\theta_{A})\in\cont(X,V)$.
Thus $adhA\leq\cont(\theta_{A})$.

For the second part, $(1)\iff(3)$ is Lemma \ref{lem:tauepsilonclosed}
for $\epsilon=0$. 

$(3)\iff(2)$: If $A$ is $r(\lambda)$-closed, and $A\in\F$ with
$x\in\lm_{\tau_{\epsilon}}\F\subset\lim_{r(\lambda)}\F$ then $x\in A$.
Thus $A$ is $\tau_{\epsilon}$-closed, that is, $\theta_{A}^{\epsilon}\in\cont(X,V)$
by Lemma \ref{lem:tauepsilonclosed}. Conversely, if $\theta_{A}^{\epsilon}\in\cont(X,V)$
for every $\epsilon>0$ then $\theta_{A}=\bigwedge_{\epsilon>0}\theta_{A}^{\epsilon}\in\cont(X,V)$.
\end{proof}

\begin{defn}
If $(X,\lambda)$ is a convergence approach space and $A\subset X$,
we call \emph{closure function of $A$, }the morphism $\cont(\theta_{A})=\cont(\adh A)\in\cont(X,V)$
and we denote it $\cl A$.
\end{defn}

Closure functions behave like topological closures in the following
sense:
\begin{prop}
\label{prop:closurelikeKuratowski} If $(X,\lambda)$ is a convergence
approach space, the operator $\cl:\mathbb{P}X\to\cont(X,V)$ satisfies
\begin{enumerate}
\item $\cl\emptyset=\theta_{\emptyset}$ is the constant zero function;
\item $A\subset B\then\cl A\leq\cl B$;
\item $\cl(A\cup B)=\cl A\vee\cl B$;
\item $\cl(\{\cl A=1\})=\cl A$.
\end{enumerate}
\end{prop}

\begin{proof}
(1) and (2) follow directly from the definition and (3) follows from
Proposition \ref{prop:Cregfunctionframe}(2) because $\theta_{A}\vee\theta_{B}=\theta_{A\cup B}$.
To see (4), note that as $A\subset\{\cl A=1\}$, so that $\cl A\leq\cl(\{\cl A=1\})$.
On the other hand, $\cl A\in\cont(X,V)$ and $\theta_{\{\cl A=1\}}\leq\cl A$
so that $\cl(\{\cl A=1\})\leq\cl A$.
\end{proof}
Note that in view of Theorem \ref{thm:adhcont}, $A\subset X$ is
$r(\lambda)$-closed if and only if $\theta_{A}=\cl A$. Hence, 
\[
\theta_{\cl_{r(\lambda)}A}=\cl(\cl_{r(\lambda)}A)\geq\cl A
\]
 but the inequality may be strict, because the function $\cl A$ may
belong to $\cont(X,V)\setminus\theta(\mathbb{P}X)$:
\begin{example}[A closure function that is not an indicator function]
\label{exa:closurenonindicator} Let $(X,d)$ be a non-trivial metric
space (say, $\mathbb{R}$ with its usual metric) then $(X,\lambda_{d})$
given by 
\[
\lambda_{d}(\F)(x)=\bigwedge_{F\in\F}\bigvee_{t\in F}d(t,x),
\]
is an approach space with $V=([0,\infty]^{op},+)$, in which 
\[
\adh A(\cdot)=d(A,\cdot)=\cl A(\cdot)
\]
 can be a closure function that is not an indicator function, as it
takes a spectrum of values (say $A=[1,2]$ then $\adh A(4)=2$).
\end{example}

Note that in view of Corollary \ref{cor:approachadh}, an approach
space $(X,\lambda)$ satisfies 
\[
\lambda(\F)(x)=\bigwedge_{B\#\F}\adh B(x)
\]
because it is a pre-approach space, and $\adh B\in\cont(X,V)$ so
that $\adh B=\cl B$. Thus, for an approach space, we obtain the following
analogs of the facts that if $\F$ is a filter on a topological space,
then 
\[
\lim\F=\bigcap_{B\#\F}\cl B\text{ and }\adh\F=\bigcap_{F\in\F}\cl F.
\]

\begin{prop}
If $(X,\lambda)$ is an approach space and $\F\in\mathbb{F}X$ then

\begin{equation}
\lambda(\F)(x)=\bigwedge_{B\#\F}\cl B(x),\label{eq:APcl}
\end{equation}
and
\begin{equation}
\adh_{\lambda}\F(\cdot)=\bigwedge_{F\in\F}\cl F(\cdot).\label{eq:APadhclosure}
\end{equation}
\end{prop}

\begin{proof}
(\ref{eq:APcl}) follows directly from the previous considerations.
To see (\ref{eq:APadhclosure}), note that for every $F\in\F$, $\adh\F\leq\adh F\leq\cl F$
because $\beta(\F)\subset\beta(F)$, so that $\adh\F\leq\bigwedge_{F\in\F}\cl F$.
On the other hand, if $\adh\F<\bigwedge_{F\in\F}\cl F$ then $\lambda\U=\bigwedge_{U\in\U}\cl U<\cl F$
for every $\U\in\beta(\F)$ and $F\in\F$, hence there is $U_{\U}\in\U$
with $\cl U_{\U}<\cl F$ for all $F\in\F$. However, there is a finite
subset $\mathbb{D}$ of $\beta(\F)$ with $\bigcup_{\U\in\mathbb{D}}U_{\U}=F_{0}\in\F$
and 
\[
\cl(F_{0})=\bigvee_{\U\in\mathbb{D}}\cl(U_{\U})
\]
by Proposition \ref{prop:closurelikeKuratowski}(3), which is a contradiction.
Hence $\adh\F=\bigwedge_{F\in\F}\cl F$.
\end{proof}
\begin{cor}
In an approach space $(X,\lambda),$
\[
\adh\F\in\cont(X,V)
\]
for every $\F\in\mathbb{F}X$.
\end{cor}

\begin{prop}
\label{prop:preimageclosure} If $(X,\lambda_{X})$ and $(Y,\lambda_{Y})$
are two convergence approach spaces and $f\in\cont(X,Y)$ and $A\subset Y$
then 
\begin{equation}
\cl_{X}(f^{-}A)\leq\cl_{Y}A\circ f.\label{eq:closurecontraction}
\end{equation}

Conversely, if $(Y,\lambda_{Y})$ is an approach space and \emph{(\ref{eq:closurecontraction})}
for every $A\subset Y$, then $f\in\cont(X,Y)$.
\end{prop}

\begin{proof}
As $\cl_{Y}A\circ f\in\cont(X,V)$ as a composite of morphisms and
$\cl_{Y}A\circ f\geq\theta_{A}\circ f=\theta_{f^{-}A}$, the conclusion
follows.

Conversely, assume $(Y,\lambda_{Y})$ is an approach space and (\ref{eq:closurecontraction})
for every $A\subset Y$, and let $\F\in\mathbb{F}X$. Since $f^{-}(A)\#\F$
for every $A\#f[\F]$, 
\[
\lambda_{X}(\F)\leq\adh_{\lambda_{X}}(f^{-}A)\leq\cl_{X}(f^{-}A)\leq\cl_{Y}A\circ f.
\]
Hence, 
\[
\lambda_{X}(\F)\leq\bigwedge_{A\#f[\F]}\cl_{Y}A\circ f=\lambda_{Y}(f[\F])\circ f
\]
because of (\ref{eq:APcl}) in $Y$.
\end{proof}
\begin{thm}
\label{thm:approachreflection} The $\ap$-reflection $\T X=(X,\lambda_{\T})$
of a convergence approach space $(X,\lambda)$ is given by
\[
\lambda_{\T}(\F)(x)=\bigwedge_{B\#\F}\cl_{\lambda}B(x).
\]
\end{thm}

\begin{proof}
Let $\lambda'$ be given by $\lambda'(\F)(x)=\bigwedge_{B\#\F}\cl_{\lambda}B(x)$
where $\cl_{\lambda}B=\cont(\theta_{B})$ in $\cont((X,\lambda),V)$.
Note that $\id_{X}\in\cont(\lambda,\lambda')$: for every $B\in\F^{\#}$,
$\lambda(\F)(\cdot)\leq\adh B(\cdot)\leq\cl_{\lambda}B(\cdot)$.

Let $(Y,\lambda_{Y})$ be an approach space and let $f\in\cont(\lambda,\lambda_{Y})$,
that is, 
\[
\lambda(\F)(x)\leq\lambda_{Y}(f[\F])(f(x))=\bigwedge_{A\#f[\F]}\cl_{Y}A(f(x)).
\]
 Then $f\in\cont(\lambda',\lambda_{Y})$. Indeed, $f^{-}(A)\#\F$
and $\cl_{\lambda}(f^{-}A)\leq\cl_{Y}A\circ f$ by Proposition \ref{prop:preimageclosure},
so that 
\begin{eqnarray*}
\lambda(\F)(x)=\bigwedge_{B\#\F}\cl_{\lambda}B(x) & \leq & \bigwedge_{A\#f[\F]}\cl_{\lambda}(f^{-}A)(x)\\
 & \leq & \bigwedge_{A\#f[\F]}\cl_{Y}A(f(x))=\lambda_{Y}(f[\F])(f(x)).
\end{eqnarray*}
Hence $\lambda'$ is the approach reflection of $\lambda$. 
\end{proof}
Note that since $V$ is an approach space and $\ap$ is a concretely
reflective subcategory of $\Cap$ with reflector $\T$, 
\begin{equation}
\cont(X,V)=\cont(\T X,V).\label{eq:sameCasT}
\end{equation}

\begin{cor}
\cite{lowe88} $V$ is initially dense in $\ap$, hence the reflective
hull of $V$ in $\mathbf{\Cap}$ is $\mathsf{Ap}$.
\end{cor}

\begin{proof}
Though the result is well-known, it is interesting to note that if
$(X,\lambda)$ is an approach space then $\{\cl_{\lambda}B:B\subset X\}$
is a specific initial family of maps to $V$. Indeed, given $\F\in\mathbb{F}X$
and $B\subset X$,
\begin{equation}
\lambda_{V}(\cl_{\lambda}B[\F])(\cl_{\lambda}B(x))\geq\cl_{\lambda}B(x),\label{eq:initial}
\end{equation}
with equality whenever $B\in\F^{\#}$. Hence, in view of (\ref{eq:APcl})
and (\ref{eq:initialCAP}),
\[
\lambda(\F)(x)=\bigwedge_{B\#\F}\cl_{\lambda}B(x)=\bigwedge_{B\subset X}\lambda_{V}(\cl_{\lambda}B[\F])(\cl_{\lambda}B(x)).
\]
To see (\ref{eq:initial}), assume $B\in\F^{\#}$ so that $\limsup_{\F}\theta_{B}=1$
by (\ref{eq:grilltheta}), and thus $\limsup_{\F}\cl B=1$ because
$\theta_{B}\leq\cl B$. Then 
\[
\lambda_{V}(\cl_{\lambda}B[\F])(\cl_{\lambda}B(x))=\cl_{\lambda}B(x)\oslash\limsup_{\F}\cl_{\lambda}B=\cl_{\lambda}B(x).
\]
On the other hand, $v\oslash t\geq v$ so that $\cl_{\lambda}B(x)\oslash\limsup_{\F}\cl_{\lambda}B\geq\cl_{\lambda}B(x)$.
\end{proof}
Therefore, given a convergence approach space $(X,\lambda)$, the
set $\cont(X,V)$ determines its $\ap$-reflection $(X,\T\lambda)$.
Specifically,
\begin{thm}
\label{thm:Clowerregframe} Given a convergence approach space $(X,\lambda)$
the map $\cont:V^{X}\to\cont(X,V)\subset V^{X}$ given by \emph{(\ref{eq:C(f)})}
is the lower-hull operator for $(X,\T\lambda)$ and the corresponding
lower regular function frame is $\cont(X,V)$. 
\end{thm}

\begin{proof}
Proposition \ref{prop:Cregfunctionframe} establishes that $\cont$
is a lower-hull operator. The corresponding lower regular function
frame is formed by the fixed points of $\cont$ by , e.g., \cite[Theorem 1.2.24]{indextheory},
that is, $\cont(X,V)$. Moreover, the corresponding approach distance
is given by (e.g., \cite[Theorem 1.2.45]{indextheory})
\[
\delta(x,A)=\bigwedge_{f\in\cont(X,V),f_{|A}=1}f.
\]

As $f_{|A}=1$ if and only if $\theta_{A}\leq f$, we conclude that
$\delta(x,A)=\cl A(x)=\adh_{\T\lambda}A(x)$, which concludes the
proof.
\end{proof}
\begin{rem}
Note that in the case of an approach space, that the lower regular
function frame is $\cont(X,V)$ is \cite[Proposition 1.3.5]{indextheory}.
Moreover, by \cite[Proposition 1.1.30]{indextheory}, $\cont$ is
completely determined by $\cl:\mathbb{P}X\to V^{X}$ via 
\[
\cont(f)(\cdot)=\bigwedge_{\epsilon<1}\vee_{i=1}^{n(\epsilon)}m_{i}^{\epsilon}\otimes\cl(M_{i}^{\epsilon})(\cdot)
\]
where $(\mu_{\epsilon}=\vee_{i=1}^{n(\epsilon)}m_{i}^{\epsilon}\otimes\theta_{M_{i}^{\epsilon}}))_{\epsilon<1}$
\emph{is a development of} $f$, that is, $\mu_{\epsilon}\geq f\geq\mu_{\epsilon}\otimes\epsilon$
for every $\epsilon<1$.
\end{rem}

\section{$\protect\ap$ as pointfree convergence spaces}

\cite{myn.ptfreeAP} offered an embedding of $\Cap$ into the category
of pointfree convergence spaces (as introduced in \cite{FredetJean})
over various categories $\mathsf{C}$ of lattices. A missing aspect
of \cite{myn.ptfreeAP} is to offer a characterization of $\ap$ within
this representation. The first part of this paper puts us in a position
to do so. Let us first recall the setting developed in \cite{FredetJean}
for pointfree convergence. 

\subsection{Recap on pointfree convergence and $\protect\Cap$ as pointfree convergence}

Let us now use for the rest of the paper a different convention: given
a lattice $L$, $\mathbb{F}L$ now denotes the set of \emph{order
theoretic }filters on $L$ (so that the set-theoretic filters on $X$
are now order theoretic filters on the powerset $\mathbb{P}X$ and
thus denoted $\mathbb{FP}X$).

In this setting, the function 
\[
\lm_{\xi}:\mathbb{FP}X\to\mathbb{P}X
\]
determining a convergence structure on $X$ is abstracted away to
a monotone function 
\[
\lim:\mathbb{F}L\to L
\]
from (order-theoretic) filters on a lattice $L$ to $L$. (\ref{eq:centeredconv})
is not part of the axiomatic in this pointfree version of convergence
spaces, though the notion can also be recovered (as so-called \emph{centered
}convergence lattices) in an abstract order-theoretic form, which
will naturally appear in the case of convergence approach spaces. 
\begin{defn}
\cite{FredetJean} Given a category $\mathsf{C}$ of lattices, a\emph{
convergence $\mathsf{C}$-object $(L,\lim)$ }is a $\mathsf{C}$-object
$L$ together with a monotone map $\lim:\mathbb{F}L\to L$. The objects
of the category $\Cconv$ are the convergence $\mathsf{C}$-objects
and the morphisms $\varphi:L\to L'$ are the $\mathsf{C}$-morphisms
that are \emph{continuous }in the sense that for every $\F\in\mathbb{F}L'$,
\begin{equation}
\lm_{L'}\F\leq\varphi(\lm_{L}\varphi^{-1}(\F)),\tag{ptfree continuity}\label{eq:ptfreeCont}
\end{equation}
 where $\varphi^{-1}(\F)=\{\ell\in L:\varphi(\ell)\in\F\}$.
\global\long\def\pt{\operatorname{pt}}%
\end{defn}

The category $\conv$ embeds coreflectively into $(\Cconv)^{op}$
when $\mathsf{C}$ is the category of frames or of coframes: the \emph{powerset
functor} $\mathbb{P}:\conv\to(\Cconv)^{op}$ sending $(X,\xi)$ to
$(\mathbb{P}X,\lm_{\xi})$ and $f:(X,\xi)\to(Y,\tau)$ to $\text{\ensuremath{\mathbb{P}f=}}f^{-1}:\mathbb{P}Y\to\mathbb{P}X$
(in $\Cconv$) is then right-adjoint to the \emph{point-functor} $\pt:(\Cconv)^{op}\to\conv$
(the coreflector) where the underlying set of $\pt L$, the set of
``points'' of $L$, is the set of $\Cconv$-morphisms from $L$
to $\mathbb{P}(1)$, hence depends on the choice of $\mathsf{C}$.
The convergence structure on $\pt L$ is given by 
\begin{equation}
\tag{Conv on pt}\lm_{\pt L}\F=\left(\lm_{L}\F^{\circ}\right)^{\bullet},\label{eq:ptconv}
\end{equation}
where $\ell^{\bullet}=\{\varphi\in\pt L:\varphi(\ell)=\{1\}\}$ and
$\F^{\circ}=\left\{ \ell\in L:\ell^{\bullet}\in\F\right\} $. Finally,
if $\varphi\in\Cconv(L,L')$ then $\pt\varphi:\pt L'\to\pt L$ is
defined by $\pt(\varphi)(f)=f\circ\varphi$.

Given a category $\mathsf{C}$ of inf-semilattices and a $\mathsf{C}$-object
$V$, the \emph{contravariant} \emph{function space functor $V^{(-)}:\mathsf{Set}\to\mathsf{C}$
}defined on object by $V^{(-)}X=V^{X}$ (ordered pointwise) and on
morphisms $f:X\to Y$ by $V^{(-)}f:V^{Y}\to V^{X}$ is defined by
$V^{(-)}(f)(h)=h\circ f$ lifts to a (contravariant) embedding functor
$V^{(-)}:\Cap\to(\Cconv)^{op}$. Namely, If $\mathsf{C}$ is a category
of inf-semilattices and $V$ is a $\mathsf{C}$-object, we say that
$\mathsf{C}$ is $V^{(-)}$\emph{-compatible }if $V^{(-)}(f)\in\mathbf{\mathsf{C}}(V^{Y},V^{X})$
whenever $f\in\mathsf{C}(X,Y)$. In particular, the categories of
lattices, of frames, of coframes, are all admissible hence $V^{(-)}$-compatible
for any choice of object $V$. Note that $\theta:\mathbb{P}X\to V^{X}$
defined by $\theta(A)=\theta_{A}$ is injective, so that we can consider
$\mathbb{P}X$ a subset of $V^{X}$.
\begin{prop}
\label{prop:Vfunctor}\cite[Proposition 4]{myn.ptfreeAP} Let $\mathsf{C}$
be a $V^{(-)}$-compatible category of complete lattices and $V$
a $\mathsf{C}$-object with $\card V\geq2$. If $(X,\lambda)$ is
a $\Cap$-object, define on $V^{X}$ the limit $\lim_{\lambda}:\mathbb{F}V^{X}\to V^{X}$
by 
\begin{equation}
\lm_{\lambda}\F=\lambda(\theta^{-}(\F)),\label{eq:limlambda}
\end{equation}
and if $f\in\Cap(X,Y)$ then $V^{(-)}(f):(V^{Y},\lim_{\lambda_{Y}})\to(V^{X},\lim_{\lambda_{X}})$
satisfies \emph{(}\ref{eq:ptfreeCont}\emph{)}. This defines a functor
$V^{(-)}:\Cap\to(\Cconv)^{op}$, which is an embedding (in the sense
of \cite{Categories}). 
\end{prop}

We will now focus on the case were $\mathsf{C}$ is a category of
frames. 
\begin{thm}
\cite[Theorem 7]{myn.ptfreeAP} If $(X,\lambda)$ is a $\Cap$-object
and $\mathsf{C}$ is the category of frames then $\pt\circ V^{(-)}(X)$
can be identified with the convergence space $(X\times V,\xi)$ where
\[
\forall v<1\;(x,v)\in\lm_{\xi}\F\iff\lambda(p_{X}[\F])(x)>v,
\]
and 
\[
(x,1)\in\lm_{\xi}\F\iff\lambda(p_{X}[\F])(x)=1.
\]
\[
\xymatrix{(\Cconv)^{op}\ar[r]^{\pt} & \conv\\
\Cap\ar[u]_{V^{(-)}}\ar[ur]_{C}
}
\]

Moreover, the functor $C=\pt\circ V^{(-)}:\Cap\to\conv$ is an embedding.
\end{thm}

Frames are pseudocomplemented. Let 
\[
a^{*}:=a\to\bot=\bigvee\{b\in L:a\wedge b=\bot\}
\]
 denote the pseudocomplement of $a$ in a frame $L$. Note also that
\cite[Appendix I, Proposition 7.1.2]{MR2868166}
\begin{equation}
(a\wedge b)^{**}=a^{**}\wedge b^{**},\label{eq:filter**}
\end{equation}
for every $a,b\in L$. As a result of (\ref{eq:filter**}), we have
\[
\F\in\mathbb{F}L\then\F^{**}:=\uparrow\{m^{**}:m\in\F\}\in\mathbb{F}L,
\]
and because $a^{***}=a^{*}$, $\F^{**}\subset\F.$ Hence, in a convergence
frame $\lim\F^{**}\leq\lim\F$.

Following \cite{convframes}, we call a convergence frame $**$-\emph{regular}
if 
\begin{equation}
\forall\F\in\mathbb{F}L,\,\lim\F=\lim(\F^{**}).\label{eq:doublestarlim}
\end{equation}

Of course, if $L$ is Boolean then $\cdot^{**}$ is the identity and
any convergence frame structure on $L$ is then $**$-regular. Note
that, despite trivializing in the Boolean case, this property is indeed
akin to regularity: a convergence space is \emph{regular }if for every
filter $\F$ the coarser filter $\adh^{\natural}\F:=\{\adh F:F\in\F\}^{\uparrow}$
has the same limit as $\F$. $**$-regularity is a similar concept
where $\adh$ is replaced by the operation $\cdot{}^{**}$. Similar
notions of regularity with respect to different operators or with
respect to a family of sets have been considered among others in \cite{DM.book}
(\footnote{The family $\{\ell\in L:\ell=\ell^{**}\}$ would play the role of
the lattice-theoretic analog of the family $\mathcal{Z}$ considered
in \cite[VII.2]{DM.book}.}). 

Given a $\mathsf{C}$-object $L$ and $\F,\G\subset L$, we say that
$\F$ and $\G$ \emph{mesh}, in symbols $\F\#\G$ if $f\wedge g\neq\bot$
for every $f\in\F$ and $g\in\G$. We call
\[
\F^{\#}=\{\ell\in L:\{\ell\}\#\F\}
\]
the \emph{grill of} $\F$. The definitions for families of subsets
of $X$ correspond to the case $L=(\mathbb{P}X,\subset)$.

Though (\ref{eq:centeredconv}) was not included in the pointfree
setting, the notion of adherence allows to recover it.

Let $(L,\lim)$ be an object of $\Cconv$. Given $\G\subset L$ the
\emph{adherence of $\G$ }(called \emph{raw adherence }in \cite{FredetJean})
is by definition
\[
\adh\G=\bigvee_{\mathbb{F}L\ni\F\#\G}\lim\F.
\]
Given $\ell\in L$, we abridge 
\[
\adh\ell:=\adh\{m\in L:\ell\leq m\}.
\]
Denoting $\mathbb{U}L$ the set of maximal filters on $L$ and, given
$\F\in\mathbb{F}L$, denoting $\beta(\F)=\{\U\in\mathbb{U}L:\U\supset\F\}$,
we have \cite{convframes}

\begin{equation}
\adh\F=\bigvee_{\mathbb{F}L\ni\H\#\F}\lim\H=\bigvee_{\mathbb{F}L\ni\H\supset\F}\lim\H=\bigvee_{\U\in\beta(\F)}\lim\U.\label{eq:adh}
\end{equation}

In particular, if $\U\in\mathbb{U}L$ then $\adh\U=\lim\U$ and $\adh\ell=\bigvee_{\ell\in\U\in\mathbb{U}L}\lim\U$.

A $\Cconv$-object $(L,\lim)$ is called \emph{centered }if $\adh\ell\geq\ell$
for every $\ell\in L$. This can be seen as a pointfree abstraction
of the point axiom (\ref{eq:centeredconv}) because an order-preserving
function $\lim:\mathbb{FP}X\to\mathbb{P}X$ satisfies the point axiom
if and only if $(\mathbb{P}X,\lim)$ is centered \cite[Remark 6.16]{FredetJean}.
It turns out that centeredness also corresponds to (\ref{eq:CAPcentered}). 

Objects of the image $V^{(-)}$-$\Cap$ of $\Cap$ in $\Cconv$ find
a simple characterization:
\begin{thm}
\label{cor:Vcapcharacterization}\cite[Corollary 17]{myn.ptfreeAP}
The objects of $V^{(-)}$-$\Cap$ are exactly the convergence frames
$(V^{X},\lim)$ that are $**$-regular and centered.
\end{thm}

In fact, if $(V^{X},\lim)$ is a $\Cconv$-object, then the function
$\lambda_{\lim}:\mathbb{FP}X\to V^{X}$ defined by
\begin{equation}
\lambda_{\lim}(\F)=\lim\theta[\F]\label{eq:lambdalim}
\end{equation}
is monotone and satisfies (\ref{eq:CAPcentered}) and $\lim_{\lambda_{\lim}}=\lim$
when $\lim$ is centered and $**$-regular. 

\subsection{$V^{(-)}$-$\protect\prap$ and $V^{(-)}$-$\protect\ap$ within
$V^{(-)}$-$\protect\Cap$}

Let us denote by $V^{(-)}$-$\prap$ and $V^{(-)}$-$\ap$ the images
of $\prap$ and $\ap$ respectively, under the functor $V^{(-)}$. 

In view of Theorem \ref{cor:Vcapcharacterization} and (\ref{eq:preAP}),
\begin{thm}
\label{thm:VPrap} The objects of $V^{(-)}$-$\prap$ are exactly
the convergence frames $(V^{X},\lim)$ that are $**$-regular, centered,
and satisfy 
\begin{equation}
\lim(\bigcap_{\D\in\mathbb{D}}\D)=\bigwedge_{\D\in\mathbb{D}}\lm\D\label{eq:Vprap}
\end{equation}
for every $\mathbb{D}\subset\mathbb{F}(V^{X})$.
\end{thm}

\begin{proof}
If $(X,\lambda)$ is an object of $\prap$ then $\lm_{\lambda}\F=\lambda(\theta^{-}(\F))$
satisfies (\ref{eq:Vprap}), because $\theta^{-}(\bigcap_{\D\in\mathbb{D}}\D)=\bigcap_{\D\in\mathbb{D}}\theta^{-}(\D)$
so that 
\[
\lm_{\lambda}(\bigcap_{\D\in\mathbb{D}}\D)=\lambda(\bigcap_{\D\in\mathbb{D}}\theta^{-}(\D))\underset{(\ref{eq:preAP})}{=}\bigwedge\lambda(\theta^{-}(\D))=\bigwedge_{\D\in\mathbb{D}}\lm_{\lambda}\D.
\]
Conversely, if $(V^{X},\lim)$ is $**$-regular, centered and satisfies
(\ref{eq:Vprap}), then $\lim=\lim_{\lambda_{\lim}}$ and 
\[
\lambda_{\lim}(\bigcap_{\D\in\mathbb{D}}\D)=\lim\theta[\bigcap_{\D\in\mathbb{D}}\D]=\lim\bigcap_{\D\in\mathbb{D}}\theta[\D]\underset{(\ref{eq:Vprap})}{=}\bigwedge_{\D\in\mathbb{D}}\lim\theta[\D]=\bigwedge_{\D\in\mathbb{D}}\lambda_{\lim}\D.
\]
\end{proof}
Since $\adh\ell\in V^{X}$ for every $\ell\in V^{X}$, we can consider
the condition $\adh\ell\in\cont(X,V)$ in terms of the structure on
$X$, which corresponds to $V^{(-)}\adh\ell$ continuous in the $\Cconv$
sense. 

In the case where $(X,\lambda)$ is a $\Cap$-space, any $f\in V^{X}$
can be tested for $f\in\cont(X,V)$ via e.g., (\ref{eq:contXtoV}),
equivalently, $V^{(-)}f:V^{V}\to V^{X}$ is a $\Cconv$-morphism,
that is, for every $\G\in\mathbb{F}V^{X}$,
\[
\lm_{\lambda}\G\leq V^{(-)}f(\lm_{\lambda_{V}}(V^{(-)}f)^{-}(\G)).
\]
Since 
\[
(V^{(-)}f)^{-}(\G)=\{m\in V^{V}:m\circ f\in\G\}
\]
and $\lim_{\lambda_{V}}\F=\lambda_{V}(\theta^{-}(\F))=\lambda_{V}(\{A\subset V:\theta_{A}\in\F\})$,
this means
\[
\lm_{\lambda}\G\leq\lambda_{V}\left(\{A\subset V:\theta_{A}\circ f\in\G\}\right)\circ f.
\]
Moreover, $\theta_{A}\circ f\in\G$ if and only if there is $g\in\G$
with $f^{-1}(V\setminus A)\subset g^{-1}(0)$. Let 
\begin{eqnarray}
f^{\star}\G & = & \{A\subset V:\theta_{A}\circ f\in\G\}\label{eq:fstarG}\\
 & =\bigcup_{g\in\G} & \{A\subset V:f^{-1}(V\setminus A)\subset g^{-1}(0)\}.\nonumber 
\end{eqnarray}

As 
\[
\lambda_{V}(\F)(v)=v\oslash\limsup_{\F}\id=\bigwedge_{A\#\F}v\oslash\bigwedge A,
\]
this means that $f\in\cont(X,V)$ when 
\[
\lm_{\lambda}\G(\cdot)\leq f(\cdot)\oslash\limsup_{f^{\star}\G}\id_{V}
\]
for every $\G\in\mathbb{F}V^{X}$. 

Moreover, by \cite[(5.12)]{myn.ptfreeAP}, if $(X,\lambda)$ is a
convergence approach space and $f\in V^{X}$, then
\begin{equation}
\adh_{\lim_{\lambda}}f=\adh_{\lambda}\{f>0\},\label{eq:adhVXtoX}
\end{equation}
where $\{f>0\}:=\{x\in X:f(x)>0\}.$

As a result,
\begin{thm}
\label{thm:VAp} The objects of $V^{(-)}$-$\ap$ are exactly the
objects $(V^{X},\lim)$ of $V^{(-)}$-$\prap$\emph{ (}$**$-regular
centered convergence frame structures satisfying \emph{(\ref{eq:Vprap}))}
that also satisfy
\begin{equation}
\lm\G(\cdot)\leq\adh f(\cdot)\oslash\limsup_{(\adh f)^{\star}\G}\id_{V}\label{eq:VAp}
\end{equation}
for every $f\in V^{X}$ and $\G\in\mathbb{F}(V^{X})$. 
\end{thm}

\begin{proof}
If $(X,\lambda)$ is an approach space then $\lim_{\lambda}$ is in
particular in $V^{(-)}$-$\prap$ and Theorem \ref{thm:VPrap} applies.
Moreover, if $f\in V^{X}$ then $\adh_{\lim_{\lambda}}f=\adh_{\lambda}\{f>0\}\in\cont(X,V)$
by Corollary \ref{cor:approachadh}. In view of the above discussion,
\[
\lm_{\lambda}\G(\cdot)\leq\adh_{\lim_{\lambda}}f(\cdot)\oslash\limsup_{(\adh_{\lim_{\lambda}}f)^{\star}\G}\id_{V}.
\]

Conversely, if $(V^{X},\lim)$ is an object of $V^{(-)}$-$\prap$,
then $\lim=\lim_{\lambda_{\lim}}$ where $\lambda_{\lim}(\F)=\lim\theta[\F]$
is a pre-approach space. If moreover (\ref{eq:VAp}) is satisfied,
then for every $A\subset X$, $\adh_{\lambda_{\lim}}A=\adh_{\lim}\theta_{A}$
by (\ref{eq:adhVXtoX}) and (\ref{eq:VAp}) ensures that $\adh_{\lambda_{\lim}}A\in\cont((X,\lambda_{\lim}),V)$
so that $(X,\lambda_{\lim})$ is an approach space by Corollary \ref{cor:approachadh}.
\end{proof}
Recall from \cite{myn.ptfreeAP,convframes} that an element $\ell$
of a convergence frame is \emph{closed }if 
\[
\ell\in\F\then\lim\F\leq\ell
\]
for every $\F\in\mathbb{F}L$. \cite[Corollary 19]{myn.ptfreeAP}
states that the closed elements of a convergence frame $(V^{X},\lim_{\lambda})$
where $(X,\lambda)$ is a convergence approach space are exactly the
functions $\theta_{A}$ where $A$ is a $r(\lambda)$-closed subset
of $X$. Hence, in view of Theorem \ref{thm:adhcont}, closed elements
of $(V^{X},\lim_{\lambda})$ are elements of $\cont(X,V)$, but not
conversely, as we have seen with Example \ref{exa:closurenonindicator}.
\bibliographystyle{plain}

\end{document}